\documentclass[
english,
]{amsart}

\usepackage{nicefrac}
\usepackage{todonotes}
\usepackage[english]{babel}
\usepackage[T1]{fontenc}
\usepackage[utf8]{inputenc} 
\usepackage[autostyle]{csquotes}
\usepackage[shortlabels]{enumitem} 
\usepackage{amsmath} 
\usepackage{amssymb} 
\usepackage{mathrsfs} 
\usepackage{textcomp}
\usepackage{gensymb} 
\usepackage{dsfont} 
\usepackage{soul}
\usepackage{faktor} 
\usepackage{mathtools} 
\usepackage{accents} 
\usepackage{verbatim} 
\usepackage{subfigure}
\usepackage{url} 
\usepackage{dsfont} 
\usepackage{float} 
\usepackage{wrapfig} 
\usepackage[format=plain,font=small,labelfont=small]{caption} 
\usepackage{graphicx} 
\usepackage{makecell} 
\usepackage[hyperfootnotes=false]{hyperref}
\usepackage[hyphenbreaks]{breakurl}
\hypersetup{
	colorlinks,
	citecolor=blue}
\usepackage[style=numeric, maxbibnames=99]{biblatex}
\usepackage{amsthm} 
\theoremstyle{plain}
\newtheorem{theorem}{Theorem}[section]
\theoremstyle{plain}
\newtheorem{corollary}[theorem]{Corollary}
\theoremstyle{plain}
\newtheorem{lemma}[theorem]{Lemma}
\theoremstyle{plain}
\newtheorem{proposition}[theorem]{Proposition}
\theoremstyle{definition}

\theoremstyle{definition}
\newtheorem{definition}[theorem]{Definition}
\theoremstyle{definition}

\theoremstyle{remark}
\newtheorem{remark}[theorem]{Remark}
\theoremstyle{remark}

\numberwithin{equation}{section} 

\allowdisplaybreaks

\begin{document}
	\newcommand{\R}{\mathbb{R}}             
\newcommand{\Comp}{\mathbb{C}}          
\newcommand{\Z}{\mathbb{Z}}             
\newcommand{\Q}{\mathbb{Q}}             
\newcommand{\N}{\mathbb{N}}             

\newcommand{\strip}{\mathbb{S}} 
\newcommand{\dd}{\,\mathrm{d}}							
\newcommand{\C}{\mathrm{C}}								
\newcommand{\Lp}[1]{\mathrm{L}^{#1}}					
\newcommand{\Lploc}[1]{\mathrm{L}^{#1}_{\mathrm{loc}}}	
\newcommand{\Wkp}[1]{\mathrm{W}^{{#1}}}					
\newcommand{\Wkploc}[1]{\mathrm{W}_{\mathrm{loc}}^{{#1}}}	
\newcommand{\Hk}[1]{\mathrm{H}^{#1}}					

\newcommand{\supp}{\text{supp}}                                                         
\newcommand{\spacedot}{\,\cdot\,}														
\newcommand{\spacebar}{\,|\,}															
\newcommand{\Aff}[2]{\mathbb{A}^{#1}(#2)}												
\newcommand{\Polyring}[2]{#1[X_1,\ldots,X_{#2}]}										
\newcommand{\limarrow}[1]{\overset{#1}{\longrightarrow}}								
\newcommand{\limu}[1]{\lim_{#1}\,}														
\newcommand{\supu}[1]{\sup_{#1}\,}														
\newcommand{\infu}[1]{\inf_{#1}\,}														
\newcommand{\liminfu}[1]{\liminf_{#1}\,}												
\newcommand{\limsupu}[1]{\limsup_{#1}\,}												
\newcommand{\esssupu}[1]{\mathrm{ess\,sup}_{#1}\,}										
\newcommand{\essinfu}[1]{\mathrm{ess\,inf}_{#1}\,}										
\newcommand{\osc}[1]{\mathrm{osc}_{#1}\,}												
\newcommand{\norm}[1]{\left\Vert #1 \right\Vert}										
\newcommand{\usedo}{uniform strongly elliptic differential operator of second order}	
\newcommand{\D}{\mathrm{D}}																
\newcommand{\DX}[1]{\mathrm{D}(#1)}														
\newcommand{\e}{\mathrm{e}}																
\newcommand{\Tt}{(T(t))_{t\geq0}}														
\newcommand{\LX}{\mathcal{L}(X)}														
\newcommand{\Cz}{$\C_0$-semigroup}														
\newcommand{\Real}{\mathrm{Re}}														
\newcommand{\Imag}{\mathrm{Im}\,}														
\newcommand{\abs}[1]{\left\vert #1 \right\vert}											
\newcommand{\PO}{\partial\Omega}														
\newcommand{\ddt}[2]{\frac{\mathrm{d}#1}{\mathrm{d}#2}}									
\newcommand{\ppt}[2]{\frac{\partial #1}{\partial #2}}									
\newcommand{\rchi}{\boldsymbol{1}}
\newcommand{\absatz}{\phantom{text}\\}													
\newcommand{\sgn}[1]{\mathrm{sgn}(#1)}													
\newcommand{\edot}[1]{\accentset{\circ}{#1}}											
\newcommand{\id}{\text{Id}}																
\newcommand{\BMO}[1]{\mathrm{BMO}(#1)}													
\newcommand{\dist}[2]{\mathrm{dist}(#1,#2)}												
\makeatletter \newcommand*{\rom}[1]{\expandafter\@slowromancap\romannumeral #1@} \makeatother		
\newcommand{\squareref}[1]{[\ref{#1}]}													
\newcommand{\ii}{\text{i}}															

\def\Xint#1{\mathchoice
	{\XXint\displaystyle\textstyle{#1}}%
	{\XXint\textstyle\scriptstyle{#1}}%
	{\XXint\scriptstyle\scriptscriptstyle{#1}}%
	{\XXint\scriptscriptstyle\scriptscriptstyle{#1}}%
	\!\int}
\def\XXint#1#2#3{{\setbox0=\hbox{$#1{#2#3}{\int}$}
		\vcenter{\hbox{$#2#3$}}\kern-.5\wd0}}
\def\ddashint{\Xint=}
\def\dashint{\Xint-}																	
\newcommand{\LBP}[1]{\mathrm{LBP}(#1)}															
    \title[Quasi-Banach complex interpolation and Calderón products]{\textbf{A note on complex interpolation of quasi-Banach function spaces}}
	\author{Moritz Egert and Benjamin W. Kosmala}
	\date{\today}
    \subjclass[2020]{46B70, 46E30}
    \keywords{Complex interpolation, quasi-Banach function space, Calderón product, $A$-convexity, separability, reiteration}
	
	\begin{abstract}
		Kalton and Mitrea characterized complex interpolation spaces of quasi-Banach function spaces as Calderón products if both interpolants are separable. We show that one separability assumption may be omitted and establish a Wolff-reiteration result with one non-separable endpoint space. 
	\end{abstract}
    \maketitle
	\section{Introduction}
    
    \noindent Recent literature uses Kalton and Mitrea's complex interpolation theory for quasi-Banach spaces \parencite{InterpolationScalesKaltonMitrea} to prove estimates for solutions to e.g.\ elliptic boundary value problems in tent spaces \parencite{EgertAuscher, AmentaAuscher}. Typically, these spaces are characterized by an integrability parameter $p\in(0,\infty]$ such that $p \in (0,1)$ yields quasi-Banach spaces, $p \in [1,\infty]$ Banach spaces and separable spaces are obtained if and only if $p\in(0,\infty)$.
    
    Amongst others, the authors of \parencite{AmentaAuscher} establish interpolation identities for the whole range $p \in (0,\infty]$, using so-called Wolff-reiteration \parencite[Thm.~2]{Wolff} as follows: They show the desired interpolation identity for parameters in $(0,\infty)$, use a duality argument on $[1,\infty)$ to cover parameters in $(1,\infty]$ and finally 'glue' $(0,\infty)$ and $(1,\infty]$ together using Wolff-reiteration.
    Regarding this final gluing procedure, \parencite[Thm.~2.12]{AmentaAuscher} contains a subtle gap, because Wolff-reiteration for the complex method is not known for general quasi-Banach spaces.

    The author of \parencite{Huang} derives such identities differently, by characterizing interpolation spaces elegantly as appropriate products and then mainly argues on the product side. 
    In their case, for instance in \parencite[Thm.~4.3]{Huang}, coincidence of products and interpolation spaces is simply taken for granted.

    In the setting of quasi-Banach function spaces, the purpose of this paper is thus to establish the identification of complex interpolation spaces as appropriate products and to derive from this the desired Wolff-type result.
    
    We first expand on Kalton and Mitrea's pioneering work \parencite[Thm.~3.4]{InterpolationScalesKaltonMitrea} by identifying complex interpolation spaces as Calderón products if only one interpolant is non-separable. Prior to our note, this was only known for sequence spaces \parencite[Sect.~7]{KaltonMayborodaMitrea}. Our (standard) notation is explained in Sections \ref{ChapterInterpolationMethod} and \ref{ChapterQBFS}.
    \begin{theorem}[Calderón formula]\label{TheoremTheBigResult}
		Let $X_0,X_1$ be $A$-convex quasi-Banach function spaces over a separable and $\sigma$-finite measure space $\Omega$, one of which is separable.
        Then $X_0+X_1$ is $A$-convex and for every $\theta \in (0,1)$ the spaces $[X_0,X_1]_\theta$ and $X_0^{1-\theta}X_1^\theta$ are separable and agree up to equivalence of quasi-norms.
	\end{theorem}
    
    We then apply a reiteration result for Calderón products due to Gomez and Milman \parencite[Thm.~4.13]{MilmanGomez}, which provides the desired Wolff-reiteration for quasi-Banach function spaces. 
    
    \begin{theorem}\label{TheoremReiterationQBFS}
		Let $X_0,X_1,X_2,X_3$ be $A$-convex quasi-Banach function spaces over a separable and $\sigma$-finite measure space $\Omega$, where either $X_0$ or $X_3$ is separable and $X_0\cap X_3$ is dense in $X_1$ and $X_2$.
        Further, let $\theta,\eta,\lambda,\mu \in [0,1]$ be parameters subject to the following conditions:
		\begin{enumerate}[(i)]
			\item $0<\theta < \eta < 1$,
			\item $\theta = \lambda \eta$,
			\item $\eta = (1-\mu)\theta+\mu$.
		\end{enumerate}
		If $X_1 = [X_0,X_2]_\lambda$ and $X_2=[X_1,X_3]_\mu$, then $X_1 = [X_0,X_3]_\theta$ and $X_2=[X_0,X_3]_\eta$ hold up to equivalence of quasi-norms.
	\end{theorem}

    \begin{remark}
        It remains an interesting open problem to prove Theorem \ref{TheoremReiterationQBFS} for more general $A$-convex quasi-Banach  spaces.
    \end{remark}

    We will recall the complex interpolation method for general quasi-Banach spaces in Section \ref{ChapterInterpolationMethod} and then focus on quasi-Banach function spaces from Section \ref{ChapterQBFS} onward, while also establishing the main tools for the proofs of the Calderón formula and the reiteration result given in Sections \ref{ChapterProofMainResult} and \ref{ChapterglimpseAtOtherDirection}.
	
	Concerning notation, all vector spaces are taken over $\Comp$, $\mathds{D}$ is the open unit disc in $\Comp$ and $\strip \coloneqq \{z \in \Comp\,|\, \Real(z) \in (0,1)\}$.
    Quasi-norms on quasi-Banach spaces $X$ will be denoted as $\norm{\spacedot}_X$, measures of measureable sets $E$ will be denoted as $\abs{E}$ and we let $\theta \in (0,1)$.

    \subsection*{Acknowledgment} This article is based on the master's thesis of the second author~\parencite{MasterThesis}, which contains further background and details on the topic. We wish to thank Emiel Lorist for guiding us through the literature on Banach function spaces and Marius Mitrea for a helpful exchange about Wolff-reiteration. 
	
	\section[Complex interpolation of quasi-Banach spaces]{Complex interpolation of quasi-Banach spaces}\label{ChapterInterpolationMethod}
	
	\noindent Let $(X_0,X_1)$ be an interpolation couple, namely $X_0$ and $X_1$ are quasi-Banach spaces that continuously embed in some common Hausdorff topological vector space $Z$.
    Following \parencite{KaltonMayborodaMitrea, InterpolationScalesKaltonMitrea}, a function $F: \strip \to X_0+X_1$ is called \textit{analytic} if for every $z_0 \in \strip$ there exists an $r \in (0,\text{dist}(z_0,\partial \strip))$ such that $F$ has a power series representation $F(z) = \sum_{k=0}^\infty f_k (z-z_0)^k$ with convergence in $X_0+X_1$ for every $z \in B_r(z_0)$.
    The space of \textit{admissible functions}, denoted as $\mathcal{F}$, contains all bounded and analytic functions, which can be continuously extended to $\overline{\strip}$ such that for $j=0,1$ the traces $t \mapsto F(j+\ii t)$ are continuous and bounded functions into $X_j$.
    This space is endowed with the quasi-norm 
    \begin{align*}
        \norm{F}_{\mathcal{F}} \coloneqq \max\left\{ \sup_{t \in \R} \norm{F(\ii t)}_{X_0}, \sup_{t \in \R} \norm{F(1 +\ii t)}_{X_1}, \sup_{z \in \strip} \norm{F(z)}_{X_0+X_1}\right\}.
    \end{align*}
    We define the \textit{complex interpolation space} as $[X_0,X_1]_\theta \coloneqq \{F(\theta)\,|\, F \in \mathcal{F}\}$ and equip it with the quotient quasi-norm $\norm{f}_{[X_0,X_1]_\theta} \coloneqq \inf\{\norm{F}_\mathcal{F}\,|\, F(\theta)=f, F \in \mathcal{F}\}$.

    Both $\mathcal{F}$ and $[X_0,X_1]_\theta$ are again quasi-Banach spaces, which agree with Calderóns method when the $X_j$ are Banach spaces \parencite[§2, §3]{Calderon}. 
    There are other complex interpolation methods for quasi-Banach spaces, such as \parencite[Ch.~3]{CalderonTorchinsky} and \parencite{MilmanGomez}, which come with their own partly undesired properties (the former does not agree with Calderóns method in the Banach setting \parencite[Ch.~1.6.7, Rem.~3]{TriebelFunctionSpacesII} and the latter does in general not produce subspaces of $X_0+X_1$, see \parencite{CwikelMilmanSagher}). They are thus not used in this paper.

    \begin{remark}
        In comparison to \parencite{InterpolationScalesKaltonMitrea}, we made two harmless but noteworthy changes:
	    \begin{enumerate}[(i)]
		    \item We dropped the assumption that $X_0\cap X_1$ is dense in both $X_j$, as it was neither needed in our arguments nor in other works in the field \parencite{Yuan, YangSickelYuan}. In particular, already \parencite{KaltonMayborodaMitrea, InterpolationScalesKaltonMitrea} contain results, for which this density assumption cannot hold. Density is typically used to extend operators $T$, which act on $X_0 \cap X_1$, first to both $X_j$ and then to $[X_0,X_1]_\theta$. However, if $T$ is already defined on both $X_j$, then it can be extended to $X_0+X_1$ simply by setting $T(x_0+x_1) \coloneqq T_0(x_1) + T_1(x_1)$.
		    \item We dropped the requirement of locally uniform convergence of the series expansion of admissible functions on $B_r(z_0)$, because convergence of quasi-Banach valued power series inside an open disc implies absolute convergence inside the same open disc, from which locally uniform convergence easily follows \parencite[Lem.~5.10]{MasterThesis}.
	    \end{enumerate}
    \end{remark}

	\section{Quasi-Banach function spaces}\label{ChapterQBFS}
 
	\noindent From now on, let $\Omega$ be a measure space. We denote with $\Lp{0}(\Omega)$ the space of $\Comp$-valued measureable functions on $\Omega$, where two functions are identified if they agree almost everywhere.
    
    In order to do function spaces right, we introduce them as in \parencite[Ch.~2]{LoristNieraeth}.
	\begin{definition}\label{DefinitionQBFS}
		A quasi-Banach space $X \subseteq \Lp{0}(\Omega)$ is called a \textit{quasi-Banach function space} or \textit{quasi-Banach function space over $\Omega$}, if:
		\begin{enumerate}[(i)]
			\item $X$ contains a \textit{weak order unit}, i.e.\ there exists $f \in X$ such that $f>0$ a.e.,
			\item $X$ satisfies the \textit{lattice property}, i.e.\ if $f \in X$, $g \in \Lp{0}(\Omega)$ with $\abs{g} \leq \abs{f}$ a.e., then $g \in X$ with $\norm{g}_X \leq \norm{f}_X$.\label{DefinitionQBFSii}
		\end{enumerate}
	\end{definition}
	
	Let $(X_0,X_1)$ be a couple of quasi-Banach function spaces. It is then admissible for complex interpolation, because the sum $X_0 + X_1$ is well-defined in $\Lp{0}(\Omega)$ and is again a quasi-Banach function space, into which $X_0$ and $X_1$ continuously embed. Also the complex interpolation space $[X_0,X_1]_\theta$ is again a quasi-Banach function space: As weak order unit we take the minimum of the weak order units in the $X_j$ and if $f \in [X_0,X_1]_\theta$ is obtained as $f=F(\theta)$ for some $F\in\mathcal{F}$, then every $g \in \Lp{0}(\Omega)$ with $\abs{g} \leq \abs{f}$ is obtained as $g = G(\theta)$ for $G\coloneqq \nicefrac{g}{f}F$ satisfying $\|G\|_{\mathcal{F}} \leq \|F\|_{\mathcal{F}}$.
	
	Having pointwise products available will allow us to characterize interpolation spaces via
	\begin{definition}\label{DefinitionCalderonProduct}
		The \textit{Calderón product} $X_0^{1-\theta}X_1^\theta$ of quasi-Banach function spaces $X_0$ and $X_1$ over $\Omega$ is defined as the set of all $f\in\Lp{0}(\Omega)$ such that
		\begin{align*}
			\norm{f}_{X_0^{1-\theta}X_1^\theta} \coloneqq \inf \left\{\norm{f_0}_{X_0}^{1-\theta}\norm{f_1}_{X_1}^\theta\,|\,\abs{f}\leq\abs{f_0}^{1-\theta}\abs{f_1}^\theta\text{ a.e.},\ f_j\in X_j\right\}
		\end{align*}
		is finite.
	\end{definition}
    
    With slight adaptations regarding the quasi-triangle inequality, completeness of Calderón products works as in \parencite[§13.5]{Calderon}, making them quasi-Banach function spaces that appear naturally in the context of interpolation.
	
	In order to show that interpolation spaces and Calderón products agree in the Banach case, Calderón used an argument based on approximation via truncations \parencite[§13.6 (i)]{Calderon}, which will also be our strategy for the quasi-Banach setting.
    This requires that truncations give rise to admissible functions and that some sort of pointwise a.e.\ convergence should be compatible with quasi-norm convergence.
	
	To identify limits of truncations in $X_0^{1-\theta}X_1^\theta$ and $[X_0,X_1]_\theta$ in the common ambient space $X_0+X_1$ later on, we will need
	\begin{proposition}\label{PropositionConvergenceXConvergenceA.E.}
		Let $X$ be a quasi-Banach function space over $\Omega$.
        If $x_n \to x$ in $X$, then $x_n \to x$ locally in measure.
        In particular, convergent sequences in $X$ admit pointwise a.e.\ converging subsequences and, consequently, limits in $X$ are determined by their pointwise a.e.\ limits.
	\end{proposition}
    
    The proof for the Banach case can be found in \parencite[Thm.~13.2]{Calderon} and only needs slight modifications to also handle the quasi-triangle inequality.
	
	Now, we can provide the classical construction of admissible functions with prescribed boundary data on $\strip$. 
    The result seems to be well known and we only include an elementary proof for completeness, since in the quasi-Banach case we cannot rely on difference quotients \parencite[p.~275]{KaltonAnalyticFunctionNonFrench}.
	\begin{lemma}\label{LemmaTruncationAdmissibleFunction}
		Let $X_0$ and $X_1$ be quasi-Banach function spaces over $\Omega$ and $f_j \in X_j$ such that for some $M>1$ their nonzero absolute values are contained in $[M^{-1},M]$ almost everywhere.
        Then $F: \strip \to X_0+X_1, z \mapsto \abs{f_0}^{1-z}\abs{f_1}^z$ is an admissible function.
	\end{lemma}
	\begin{proof} The proof comes in four steps.\\
 
        Step 1: $F$ is well-defined.\\

        \noindent Let $z \in \strip$. Since powers of $0$ are not unanimously defined, we need to be careful when defining $F$. We set $\abs{f_0}^{1-z} \coloneqq g_z \circ \abs{f_0}$, where
        \begin{align*}
            g_z : [0,\infty) \to \Comp, \qquad w \mapsto \begin{cases}
                w^{1-z}, &\text{if\ } w \neq 0,\\
                0, &\text{else}.
            \end{cases}
        \end{align*}
        Then $g_z$ is measurable and using a similar definition for $\abs{f_1}^z$, defining $F(z) \coloneqq \abs{f_0}^{1-z}\abs{f_1}^z$ yields a measurable function. By Young's inequality, we have that
		\begin{align*}
			\abs{\abs{f_0}^{1-z}\abs{f_1}^z} = \abs{f_0}^{\Real(1-z)}\abs{f_1}^{\Real(z)} \leq \Real(1-z)\abs{f_0} + \Real(z) \abs{f_1}.
		\end{align*}
		Thus, $F(z) \in X_0+X_1$ by the lattice property of $X_0+X_1$.\\

        Step 2: $F$ is bounded on $\strip$.\\
        
        \noindent The pointwise estimate for $F$ on $\strip$ directly translates into a quasi-norm estimate:
        \begin{align}\label{EquationRoughBoundAnalyticFunction}
            \norm{F(z)}_{X_0+X_1} \leq \norm{\Real(1-z)\abs{f_0} + \Real(z) \abs{f_1}}_{X_0+X_1}
					&\leq \norm{f_0}_{X_0} + \norm{f_1}_{X_1}.
        \end{align}
   
        Step 3: $F$ extends continuously to $\overline{\strip}$ such that the traces are bounded into $X_j$.\\

        \noindent To define an extension, we only need to observe that the definition of $g_z$ can be extended to $z \in \overline{\strip}$ and still yields a measurable function. Abusing notation, we denote this extension of $F$ to $\overline{\strip}$ again as $F$. Then $\abs{F(j + \ii t)} = \abs{f_j}$ and we obtain that $t \mapsto F(j + \ii t)$ is a bounded map into $X_j$.
        
        For continuity on $\overline{\strip}$, let $z,z' \in \overline{\strip}$ and denote with $\overline{zz'}$ the line segment connecting $z$ and $z'$. On $A \coloneqq \supp(f_0)\cap\supp(f_1)$, we obtain pointwise a.e.\ the estimate
		\begin{align*}
			& \vert F(z) - F(z')\vert \\
			=&\abs{\abs{f_0}^{1-z}\abs{f_1}^{z} - \abs{f_0}^{1-z'}\abs{f_1}^{z'}} \\
			=&\abs{\int_{\overline{zz'}} \log(\abs{f_1}/\abs{f_0}) \abs{f_0}^{1-\xi} \abs{f_1}^\xi\dd \xi} \\
			\leq &\abs{z-z'} \log(\abs{f_1}/\abs{f_0}) \sup_{\xi \in \overline{zz'}} \abs{f_0}^{1-\Real(\xi)}\abs{f_1}^{\Real(\xi)}\\
			\leq &\abs{z-z'} \log(M^2) M.
		\end{align*}
        The same estimate holds true on $\Omega\setminus A$. As $\rchi_A \leq M^{-1} \abs{f_j}$, the lattice property of $X_0+X_1$ yields $\rchi_A \in X_0+X_1$ with
        \begin{align*}
            \norm{F(z) - F(z')}_{X_0+X_1} \leq \abs{z-z'} \log(M^2) M \norm{\rchi_A}_{X_0+X_1},
        \end{align*}
        showing continuity of $F$ up to $\overline{\strip}$. The continuity of the traces works similarly by taking $X_j$-norms when $z,z' \in j + \ii \R$.\\

        Step 4: $F$ is analytic.\\
        
        \noindent Let $z_0 \in \strip$. By the series expansion of the exponential, we obtain for all $z \in \Comp$ pointwise a.e.\ on $A$ that
		\begin{align*}
			F(z) &= \abs{f_0}^{1-z}\abs{f_1}^z \\
			&= \abs{f_0} \e^{z_0 \log\left(\abs{\frac{f_1}{f_0}}\right)} \e^{(z-z_0)\log\left(\abs{\frac{f_1}{f_0}}\right)} \\
			&= \sum_{k=0}^\infty \frac{\abs{f_0}}{k!} \e^{z_0 \log\left(\abs{\frac{f_1}{f_0}}\right)}\log^k\left(\abs{\frac{f_1}{f_0}}\right)  (z-z_0)^k \\
            &\eqqcolon \sum_{k=0}^\infty g_k (z-z_0)^k.
		\end{align*}
		We again denote the extension of $g_k$ by 0 to $\Omega$ as $g_k$. This yields measurable functions and $g_k \in X_0+X_1$, as $k! \abs{g_k} \leq \abs{f_0}M^2 [\log(M^2)]^{k}$ a.e.\ and the latter is in $X_0$.
        By Proposition \ref{PropositionConvergenceXConvergenceA.E.}, this is the only candidate for a series expansion in $X_0+X_1$ and we need to check that the series converges in $X_0+X_1$.
        If $z \in \strip$, then
		\begin{align*}
			\norm{\sum_{k=n}^m g_k (z-z_0)^k}_{X_0+X_1}
			\leq &\norm{\sum_{k=n}^m \frac{\abs{f_0}}{k!} M^2 \log^k(M^2)\abs{z-z_0}^k}_{X_0+X_1} \\
			\leq &M \norm{\rchi_{A}}_{X_0+X_1} M^2 \sum_{k=n}^m \frac{\log^k(M^2)}{k!} \abs{z-z_0}^k
		\end{align*}
		shows that the series converges indeed in $X_0+X_1$ for every $z \in B_{\text{dist}(z_0,\partial\strip)}(z_0)$.
	\end{proof}
	To relate functions to their truncations, we use the following concept.
	\begin{definition}\label{DefinitionOrderContinuity}
		A quasi-Banach function space $X$ over $\Omega$ is \textit{order continuous}, if for every nonnegative sequence $(f_n)_n$ in $X$ with $f_n \searrow 0$ a.e.\ it follows that $f_n \to 0$ in $X$.
	\end{definition}
	
	We have access to order continuity if $\Omega$ is `reasonably separable'.
    \begin{definition}[{{\parencite[p.~12]{LoristNieraeth}}}]
        The measure space $\Omega$ is called \textit{separable}, if there exists a sequence of measurable sets $(A_n)_n$ in $\Omega$ such that for every measurable set $E \subseteq \Omega$ with $\abs{E}<\infty$ and every $\varepsilon>0$ there exists an $n \in \N$ such that $\abs{A_n \Delta E}< \varepsilon$.
    \end{definition}
    In particular, if $\Omega$ is a separable metric space, which is endowed with its Borel $\sigma$-algebra and a regular measure, $\Omega$ is separable in the above sense. (Take $(A_n)_n$ as an enumeration of finite unions of balls with rational radius centered around a dense sequence in $\Omega$.)
	
	\begin{theorem}\label{TheoremCharacterizationSeparabilityOrderContinuity}
		Let $X$ be a quasi-Banach space over a separable and $\sigma$-finite measure space $\Omega$. Then $X$ is separable if and only if $X$ is order continuous.
	\end{theorem}
	
    Showing that order continuity implies separability is done similarly to the $\Lp{p}$-case by making use of the weak order unit, see \parencite[Prop.~3.13]{LoristNieraeth}.
    The other direction is more involved: Arguing by contraposition boils down to constructing a discrete copy of $\mathcal{P}(\N)$ in $X$ as in \parencite[Cor.~14.5]{AliprantisBurkinshaw}, which can be adapted to the quasi-Banach setting.

    Our generalization of the Calderón formula is based on
    
	\begin{proposition}\label{PropositionOrderContinuityProduct}
		Let $X_0$ and $X_1$ be quasi-Banach function spaces over $\Omega$. If either of them is order continuous, then so is $X_0^{1-\theta}X_1^\theta$.
	\end{proposition}
    
    The proof in \parencite[Thm.~1.29]{OrderContinuityCalderonProduct} relies only on pointwise a.e.\ estimates and the lattice property and thus works \emph{verbatim} in the quasi-Banach setting.

    By Theorem \ref{TheoremCharacterizationSeparabilityOrderContinuity}, the space $X_0^{1-\theta}X_1^\theta$ is order continuous (and therefore truncation arguments work) if $X_0^{1-\theta}X_1^\theta$ is separable, which is the case if for instance $X_0$ and $X_1$ are separable, as is easily seen from the definition of $\norm{\,\cdot\,}_{X_0^{1-\theta}X_1^\theta}$. By Proposition~\ref{PropositionOrderContinuityProduct} and Theorem \ref{TheoremCharacterizationSeparabilityOrderContinuity}, we can reduce to a single separability assumption.
    	
	We will need further properties that a quasi-Banach function space may have. The first one is a weak version of the maximum principle.
    \begin{definition}
        A quasi-Banach function space $X$ over $\Omega$ is called \textit{A-convex} or \textit{analytically convex}, if there exists $C\geq 1$ such that for all polynomials $P:\overline{\mathds{D}} \to X$ it follows that $\norm{P(0)}_X \leq C \sup_{\abs{z}=1}\norm{P(z)}_X$.
    \end{definition}

    \begin{remark}
        Tent spaces as in \parencite{AmentaAuscher, Huang, CoifmanMeyerStein} are covered by our definitions of quasi-Banach function spaces and $A$-convexity, see e.g.\ the proof of \parencite[Thm.~4.3]{Huang}. Additionally, general Banach spaces are $A$-convex \parencite[p.~21]{KaltonMayborodaMitrea}.
    \end{remark}
	
	The second one will yield a different perspective on the proof of the Calderón formula later on.
	\begin{definition}\label{DefinitionFatouProperty}
		A quasi-Banach function space $X$ over $\Omega$ has the \textit{weak Fatou property} if there exists $C \geq 1$ such that for all sequences of non-negative functions $(f_n)_n$ in $X$  that satisfy $\liminf_{n\to\infty} \norm{f_n}_X < \infty$, it follows that $\liminf_{n\to\infty}f_n \in X$ with $\norm{\liminf_{n\to\infty}f_n}_X \leq C \liminf_{n \to \infty} \norm{f_n}_X$.
	\end{definition}
	
        Among the different versions of the Fatou property \parencite[Ch.~15, §65]{Zaanen} we choose this one, because it is most reminiscent of Fatou's lemma from integration theory.
	
	\section{Proof of the Calderón formula}\label{ChapterProofMainResult}
	
    \noindent We recall the Aoki-Rolewicz Theorem \parencite[Thm.~1.3]{FSpaceSampler}, which asserts that for every quasi-Banach space $X$ there is an $r\in(0,1]$ and an equivalent quasi-norm, which is $s$-subadditive and continuous for all $s\in(0,r]$.
    In particular, $\norm{\sum_{k=0}^n x_k}_X^s \leq C^2 \sum_{k=0}^n \norm{x_k}_X^s$ holds whenever $x_0,\ldots,x_n \in X$ and $s\in(0,r]$, where $C\geq 1$ is the constant from the quasi-norm equivalence.
    
    For the following proof, we may assume that the equivalent quasi-norms on either of $[X_0,X_1]_\theta$ and $X_0^{1-\theta}X_1^\theta$ are $s$-subadditive for the same value $r \in(0,1]$ and switching to them from the original norm will come at the cost of a constant $C \geq 1$ or $D\geq1$, respectively. As only the equivalent quasi-norm on $X_0^{1-\theta}X_1^\theta$ shows up explicitly, this one will be denoted as $\norm{\,\cdot\,}$. We also denote the constant from the quasi-triangle inequality of $[X_0,X_1]_\theta$ as $C_\square$.
    
	\begin{proof}[Proof of Theorem \ref{TheoremTheBigResult}]
		We use ideas from \parencite[§13.6 (i)]{Calderon} to show that the inclusion $X_0^{1-\theta}X_1^\theta \subseteq [X_0,X_1]_\theta$ holds, that it is bounded and that both spaces are separable.
  
        That $X_0+X_1$ is $A$-convex and that there is a bounded inclusion $[X_0,X_1]_\theta \subseteq X_0^{1-\theta}X_1^\theta$ was proven in \parencite[Thm.~3.4]{InterpolationScalesKaltonMitrea} without the assumption on separability. To convince the reader that these results indeed do not rely on separability, an overview on their proofs can be found in Section~\ref{ChapterglimpseAtOtherDirection}. We set $K \coloneqq 2 D^2$ and work with the following auxiliary set
		\begin{align*}
			\mathcal{D} \coloneqq \Bigl\{f \in X_0^{1-\theta}X_1^\theta\,|\, &\exists\, M>1,\text{ s.t. }\abs{f}=K \norm{f}_{X_0^{1-\theta}X_1^\theta}\abs{f_0}^{1-\theta}\abs{f_1}^\theta,\\
			&f_j \in X_j,\, \norm{f_j}_{X_j} \leq 1,\ M^{-1} \leq \abs{f_j} \leq M\text{ a.e.\ on } \supp(f_j)\Bigr\}.
		\end{align*}
        The actual argument comes in three steps.
            \medskip
		
		Step 1: The bounded inclusion $\mathcal{D} \subseteq [X_0,X_1]_\theta$.\\
		
		\noindent If $f \in \mathcal{D}$, then
		\begin{align*}
			F: \strip \to X_0 + X_1, \qquad z \mapsto \frac{f}{\abs{f}} \abs{f_0}^{1-z}\abs{f_1}^z
		\end{align*}
		is admissible by Lemma \ref{LemmaTruncationAdmissibleFunction}.
        In view of \eqref{EquationRoughBoundAnalyticFunction}, the bound $\norm{F}_\mathcal{F} \leq 2$ is immediate and as $F(\theta) = f(K\norm{f}_{X_0^{1-\theta}X_1^\theta})^{-1}$, we obtain $f \in [X_0,X_1]_\theta$ and 
		\begin{align} \label{equation:Step1_estimate}
			\norm{f}_{[X_0,X_1]_\theta} \leq \norm{F}_\mathcal{F} K\norm{f}_{X_0^{1-\theta}X_1^\theta} \leq  2K\norm{f}_{X_0^{1-\theta}X_1^\theta}.
		\end{align}

            \medskip
		
		Step 2: $\mathcal{D}$ is dense in $X_0^{1-\theta}X_1^\theta$.\\
		
		\noindent By assumption, either $X_0$ or $X_1$ is separable, so that one is order continuous by Theorem \ref{TheoremCharacterizationSeparabilityOrderContinuity}.
        Using Proposition \ref{PropositionOrderContinuityProduct}, we obtain order continuity of $X_0^{1-\theta}X_1^\theta$ and then separability, again by Theorem \ref{TheoremCharacterizationSeparabilityOrderContinuity}.
		
		Let $f \in X_0^{1-\theta}X_1^\theta$. Of course, we can assume $f \neq 0$.
        Then there exist $f_j \in X_j \setminus \{0\}$ such that
		\begin{align*}
			\abs{f}\leq\abs{f_0}^{1-\theta}\abs{f_1}^\theta, \quad \norm{f_0}_{X_0}^{1-\theta}\norm{f_1}_{X_1}^\theta \leq \frac{3}{2}\norm{f}_{X_0^{1-\theta}X_1^\theta}.
		\end{align*}
		Replacing $f_0$ by $\abs{\nicefrac{f}{\abs{f_1}^\theta}}^{\nicefrac{1}{1-\theta}} \leq \abs{f_0}$, we may assume that $\abs{f}=\abs{f_0}^{1-\theta}\abs{f_1}^\theta$.  For $n \in \N$ we introduce $E_n\coloneqq \{ \omega \in \Omega\,|\, \nicefrac{1}{n} \leq \abs{f_0(\omega)},\abs{f_1(\omega)} \leq n\}$ and define truncations via $g_n \coloneqq f \rchi_{E_n}$.
        Then $g_n \in X_0^{1-\theta}X_1^\theta$ and order continuity yields $g_n \to f$ in $X_0^{1-\theta}X_1^\theta$ due to $\abs{f-g_n} \searrow 0$.
		
		In order to show $g_n \in \mathcal{D}$, we observe that $g_n \to f$ in $X_0^{1-\theta}X_1^\theta$ implies $\norm{g_n} \to \norm{f}$ by continuity of $\norm{\,\cdot\,}$ and so we have $\norm{f} \leq \nicefrac{4}{3}\norm{g_n}$ for all large enough $n$. Making $n$ even larger, we can also assume $\norm{f_j \rchi_{E_n}}_{X_j} > 0$ and define $h_{i,n} \coloneqq f_i \rchi_{E_n} \Vert f_i \rchi_{E_n}\Vert^{-1}$ for those $n$. We have $h_{i,n} \in X_i$ and obtain
		\begin{align*}
			\abs{g_n} &= \norm{f_0 \rchi_{E_n}}_{X_0}^{1-\theta} \norm{f_1 \rchi_{E_n}}_{X_1}^{\theta} \abs{h_{0,n}}^\theta \abs{h_{1,n}}^{1-\theta}\notag\\
			&\leq \frac{3}{2} \norm{f}_{X_0^{1-\theta}X_1^\theta} \abs{h_{0,n}}^{1-\theta}\abs{h_{1,n}}^\theta \\
			&\leq \frac{3}{2}\, D \norm{f}_{} \abs{h_{0,n}}^{1-\theta}\abs{h_{1,n}}^\theta \notag\\
			&\leq 2 D \norm{g_n}_{} \abs{h_{0,n}}^{1-\theta}\abs{h_{1,n}}^\theta \notag\\
			&\leq 2 D^2 \norm{g_n}_{X_0^{1-\theta}X_1^\theta} \abs{h_{0,n}}^{1-\theta}\abs{h_{1,n}}^\theta \notag\\
            &= K \norm{g_n}_{X_0^{1-\theta}X_1^\theta} \abs{h_{0,n}}^{1-\theta}\abs{h_{1,n}}^\theta. \notag
		\end{align*}
		Replacing $h_{0,n}$ by $\vert g_n (K \norm{f_n}_{X_0^{1-\theta}X_1^\theta} \abs{h_{1,n}}^\theta )^{-1}\vert^{1/(1-\theta)} \leq \abs{h_{0,n}}$ achieves equality and we see that every $g_n$ belongs to $\mathcal{D}$ by making a certain choice of $M > 1$.\\
		
		Step 3: $X_0^{1-\theta}X_1^\theta \subseteq [X_0,X_1]_\theta$ follows by density. \\
  
        \noindent This is not immediate, as $\mathcal{D}$ might not be a linear subspace. So, let $f \in X_0^{1-\theta}X_1^\theta$.
        Inductively, we shall construct a sequence $(g_n)_n$ in $ \mathcal{D}$ such that
		\begin{align}
			\Bigg\Vert f-\sum_{k=0}^n f_k \Bigg\Vert^r &\leq \rho^{n+1} \norm{f}^r, \label{EquationDensitySeriesEstimate}\\
			\norm{f_n}^r& \leq \rho^{n+1} \norm{f}^r \label{EquationDensitySequenceEstimate}
		\end{align}
		hold for some $\rho \in (0,1)$.
		
		For $n=0$, let $c,\varepsilon>0$ be parameters to be chosen momentarily. 
        By the previous step, there is an $f_0 \in \mathcal{D}$ such that $\norm{cf-f_0}^r \leq \varepsilon \norm{f}^r$.
        This function also satisfies
		\begin{align*}
			\norm{f-f_0}^r \leq \norm{cf - f_0}^r + \norm{(1-c)f}^r \leq (\varepsilon+(1-c)^r) \norm{f}^r
        \end{align*}
		and
        \begin{align*}
			\norm{f_0}^r \leq \norm{f_0-cf}^r + \norm{cf}^r \leq (\varepsilon + c^r) \norm{f}^r.
		\end{align*}
		Fixing $c=\nicefrac{1}{2}$ and $\varepsilon$ so that $\rho\coloneqq \varepsilon + c^r <1$ yields the claim for $n=0$.
        The induction over $n$ is done similarly by replacing $f$ with $f-\sum_{k=0}^nf_k$ and $f_0$ with $f_{n+1}$.
		
		Since $\rho\in(0,1)$, estimate \eqref{EquationDensitySeriesEstimate} implies that $\sum_{k=0}^n f_k \limarrow{n \to \infty} f$ in $X_0^{1-\theta}X_1^\theta$.
        By \eqref{equation:Step1_estimate} and \eqref{EquationDensitySequenceEstimate} we also obtain
		\begin{align*}
			\norm{f_n}^r_{[X_0,X_1]_\theta} &\leq (2K)^r \norm{f_n}^r_{X_0^{1-\theta}X_1^\theta}\\
			&\leq (2K)^r D^r \norm{f_n}^r_{}\\
			&\leq (2K)^r D^r \rho^{n+1} \norm{f}^r_{} \\
			&\leq (2K)^r D^{2r} \rho^{n+1} \norm{f}^r_{X_0^{1-\theta}X_1^\theta},
		\end{align*}
		which we use to show that $(\sum_{k=0}^n f_k)_n$ is a Cauchy sequence in $[X_0,X_1]_\theta$:
		\begin{align*}
			\Bigg\Vert\sum_{k=m}^n f_k\Bigg\Vert^r_{[X_0,X_1]_\theta}
            \leq &C^{2r} \sum_{k=m}^n \norm{f_k}^r_{[X_0,X_1]_\theta} \\
			\leq &C^{2r} (2K)^r D^{2r} \sum_{k=m}^n \rho^{k+1} \norm{f}^r_{X_0^{1-\theta}X_1^\theta}.
		\end{align*}
		Thus, the limit $\widetilde{f}$ in $[X_0,X_1]_\theta$ satisfies
		\begin{align*}
			\Vert\widetilde{f}\Vert_{[X_0,X_1]_\theta} 
            \leq &C_{\square} \Bigg(\Bigg\Vert\widetilde{f}-\sum_{k=0}^n f_k\Bigg\Vert_{[X_0,X_1]_\theta} + \Bigg\Vert\sum_{k=0}^n f_k\Bigg\Vert_{[X_0,X_1]_\theta}\Bigg) \\
			\leq &C_{\square}\Bigg(\Bigg\Vert\widetilde{f}-\sum_{k=0}^n f_k\Bigg\Vert_{[X_0,X_1]_\theta} + C^2 (2K) D^{2} \Bigg(\sum_{k=0}^n\rho^{k+1}\Bigg)^{\nicefrac{1}{r}} \norm{f}_{X_0^{1-\theta}X_1^\theta}\Bigg)
		\end{align*}
		and so by passing to the limit, we obtain
		\begin{align}\label{EquationContinuousInclusionProductInterpolationSpace}
			\Vert\widetilde{f}\Vert_{[X_0,X_1]_\theta} \leq C_{\square} C^2 (2K) D^{2} \left(\frac{\rho}{1-\rho}\right)^{\nicefrac{1}{r}}\norm{f}_{X_0^{1-\theta}X_1^\theta}.
		\end{align}
		We have shown that in the limit as $n\to\infty$, the partial sums above converge to $f$ in $X_0^{1-\theta}X_1^\theta$ and to $\widetilde{f}$ in $[X_0,X_1]_\theta$. Since both spaces continuously embed into $X_0+X_1$ (by definition and Young's inequality, respectively), the limits can be identified by virtue of Proposition~\ref{PropositionConvergenceXConvergenceA.E.}.
        Thus, \eqref{EquationContinuousInclusionProductInterpolationSpace} yields the claim.\qedhere
	\end{proof}
	
        Our second main result follows immediately:
	
	\begin{proof}[Proof of Theorem \ref{TheoremReiterationQBFS}]
		If $X_0$ is separable, then by Theorem \ref{TheoremTheBigResult} also $X_1$ is separable with $X_1 = [X_0,X_2]_\lambda = X_0^{1-\lambda}X_2^\lambda$.
        Since $X_2=[X_1,X_3]_\mu$, separability of $X_1$ yields $X_2=X_1^{1-\mu}X_3^\mu$.
        The same conclusion can be achieved when $X_3$ is separable.
        Now, \parencite[Thm.~4.13]{MilmanGomez} applies and yields $X_1 = X_0^{1-\theta}X_3^\theta$ and $X_2=X_0^{1-\eta}X_3^\eta$. Using Theorem \ref{TheoremTheBigResult} again to convert Calderón products back to interpolation spaces, we finally obtain $X_1 = [X_0,X_3]_\theta$ and $X_2=[X_0,X_3]_\eta$.
	\end{proof}
 
    Our proof of the inclusion $X_0^{1-\theta}X_1^\theta \subseteq [X_0,X_1]_\theta$ uses separability of the Calderón product in a crucial way. We can simplify the argument, whenever more structure of the interpolation space is known a priori.
	\begin{corollary}
		Let $X_0,X_1$ be $A$-convex quasi-Banach function spaces over a $\sigma$-finite measure space $\Omega$ and let $[X_0,X_1]_\theta$ satisfy the weak Fatou property.
        Then $X_0+X_1$ is $A$-convex and the spaces $[X_0,X_1]_\theta$ and $X_0^{1-\theta}X_1^\theta$ agree up to equivalence of quasi-norms.
	\end{corollary}
    \begin{proof}
        We took inspiration from the proof of the corollary in \parencite[p.~242]{KreinPetuninSemenov}. In our proof of $X_0^{1-\theta}X_1^\theta \subseteq [X_0,X_1]_\theta$ above, separability is only used to prove that truncations also converge in $X_0^{1-\theta}X_1^\theta$. If $[X_0,X_1]_\theta$ satisfies the Fatou property, we can manage this inclusion differently.
	
	    As a preliminary result, we note that $f\in\Lp{0}(\Omega)$ belongs to some quasi-Banach function space $X$ if and only if $\abs{f} \in X$ and in this case we have $\norm{f}_X = \norm{\,\abs{f}\,}_X$. This follows by the lattice property, since $\abs{f}\leq \abs{\,\abs{f}\,}$ and $\abs{\,\abs{f}\,}\leq \abs{f}$. Thus, to show that $X_0^{1-\theta}X_1^\theta \subseteq [X_0,X_1]_\theta$ and that the corresponding quasi-norm estimates hold, it suffices to consider absolute values.
     
        Let $f \in X_0^{1-\theta}X_1^\theta$. Again, we may assume $f\neq0$. Then there exist $f_j \in X_j\setminus\{0\}$ such that
	    \begin{align*}
		      \abs{f}\leq \abs{f_0}^{1-\theta}\abs{f_1}^\theta, \quad \norm{f_0}_{X_0}^{1-\theta}\norm{f_1}_{X_1}^\theta \leq 2 \norm{f}_{X_0^{1-\theta}X_1^\theta}
	    \end{align*}
        and as in Step 2 of the proof of Theorem \ref{TheoremTheBigResult} we may assume that $\abs{f} = \abs{f_0}^{1-\theta}\abs{f_1}^\theta$.
	    For $n \in \N$ set $E_n\coloneqq \{ \omega \in \Omega\,|\, \nicefrac{1}{n} \leq \abs{f_0(\omega)},\abs{f_1(\omega)} \leq n\}$ and define $g_n \coloneqq f \rchi_{E_n}$. The maps
	    \begin{align*}
		      F_n: \strip \to X_0+X_1,\qquad z \mapsto \abs{\frac{f_0\rchi_{E_n}}{\norm{f_0\rchi_{E_n}}_{X_0}}}^{1-z}\abs{\frac{f_1\rchi_{E_n}}{\norm{f_1\rchi_{E_n}}_{X_1}}}^{z}
	    \end{align*}
	    are well-defined for large enough $n$ and admissible by Lemma \ref{LemmaTruncationAdmissibleFunction} and so $\abs{g_n} \in [X_0,X_1]_\theta$ with
	    \begin{align*}
		      \norm{\,\abs{g_n}\,}_{[X_0,X_1]_\theta} &\leq \norm{F_n}_\mathcal{F}  \norm{f_0\rchi_{E_n}}_{X_0}^{1-\theta} \norm{f_1\rchi_{E_n}}_{X_1}^\theta \\
		      &\leq 2 \norm{f_0}_{X_0}^{1-\theta} \norm{f_1}_{X_1}^\theta  \\
		      &\leq 4 \norm{f}_{X_0^{1-\theta}X_1^\theta},
	    \end{align*}
        where we have used \eqref{EquationRoughBoundAnalyticFunction} and the lattice property in the second step.
	    Since $\abs{f} = \liminf_n \abs{g_n}$ holds with $\liminf_n\norm{\,\abs{g_n}\,}_{[X_0,X_1]_\theta}<\infty$, the weak Fatou property of $[X_0,X_1]_\theta$ now implies that $\abs{f} \in [X_0,X_1]_\theta$ with
	    \begin{align*}
		    \norm{\,\abs{f}\,}_{[X_0,X_1]_\theta} \leq C \liminf_{n\to\infty}\norm{\,\abs{g_n}\,}_{[X_0,X_1]_\theta} \leq 4C \norm{f}_{X_0^{1-\theta}X_1^\theta},
	    \end{align*}
	    where $C\geq1$ is the constant from the weak Fatou property of $[X_0,X_1]_\theta$. By the preliminary remark, we see that $f \in [X_0,X_1]_\theta$ with $\norm{f}_{[X_0,X_1]_\theta} \leq 4C \norm{f}_{X_0^{1-\theta}X_1^\theta}$.
     
        The remainder of this proof now works as indicated in Section \ref{ChapterglimpseAtOtherDirection} below.
    \end{proof}
    
	\section{A glimpse at the other inclusion of the Calderón formula}\label{ChapterglimpseAtOtherDirection}
    \noindent In order to fully justify why separability is only used for one direction of Theorem~\ref{TheoremTheBigResult}, we give a short overview on the ingredients of the proof of the inclusion $[X_0,X_1]_\theta \subseteq X_0^{1-\theta}X_1^\theta$ in \cite{InterpolationScalesKaltonMitrea}.
    
	A motivation for the Banach case can be found in \parencite[Chap.~4]{KreinPetuninSemenov}, where it is argued that admissible functions $F$ are determined by their boundary values through the Poisson formula
	\begin{align}\label{EquationPoissonKernelBanach}
		F(z)&=\int_\R F(\ii t) P_0(z,\ii t) \dd t + \int_\R F(1+\ii t) P_1(z,1+\ii t) \dd t\\
		&\eqqcolon (1-\text{Re}(z))x_0 + \text{Re}(z)x_1 \notag,
	\end{align} 
	where $P_0$ and $P_1$ denote the components of the Poisson kernel for $\overline{\strip}$, while $x_0$ and $x_1$ define elements in $X_0$ and $X_1$ respectively.
    The above equality can be turned into a product estimate of the form $\abs{F(\theta)} \leq x_0^{1-\theta}x_1^\theta$ a.e., showing $[X_0,X_1]_\theta \subseteq X_0^{1-\theta}X_1^\theta$ with continuous inclusion due to the validity of the triangle inequality for $X_j$-valued integrals.
	
	As a preliminary result in the quasi-Banach case, Kalton showed in \parencite[Thm.~4.4]{AConvexIffLConvex} and \parencite[Thm.~2.2]{LConvexIffpConvex} that for a quasi-Banach function space $X$, the property of $A$-convexity is equivalent to so-called \textit{$p$-convexity}, which means that there is $p\in(0,1]$ and $C\geq1$ depending only on $p$ such that for all $g_0,\ldots,g_n \in X$ we have
	\begin{align*}
		\norm{\left(\sum_{k=0}^n \abs{g_k}^p\right)^{1/p}}_X \leq C \left(\sum_{k=0}^n \norm{g_k}_X^p\right)^{1/p}.
	\end{align*}
	In \parencite[Thm.~2.2]{LConvexIffpConvex}, it is also shown that $X$ is $p$-convex if and only if $X$ is $q$-convex for all $q \in (0,p]$.
 
    Straightforward estimates then show that $X_0+X_1$ is $p$-convex \parencite[Thm.~3.4]{InterpolationScalesKaltonMitrea}.
    In particular, we may assume $X_0$, $X_1$ and $X_0+X_1$ to be $p$-convex for the same range of parameters $q \in (0,p]$.
	
	First, in \parencite[Thm.~3.4]{InterpolationScalesKaltonMitrea} it is used that $p$-convex quasi-Banach function spaces $X$ admit the \textit{$p$-convexification} $X^p \coloneqq \{f \in \Lp{0}(\Omega)\,|\,\abs{f}^{1/p} \in X\}$, which is endowed with $\norm{f}_{X^p} \coloneqq \Vert\abs{f}^{1/p}\Vert_X^p$ and after renorming satisfies a triangle inequality \parencite[Prop. 3.23]{MasterThesis}, is a Banach function space.
    The idea of the proof is then to mimic equation \eqref{EquationPoissonKernelBanach} for $\abs{F}^q$, which -- and this is a remarkable observation -- can be turned back into an estimate for $\abs{F}$ via a limiting process as $q \to 0$.
	
	Indeed, if $f=F(\theta)\in[X_0,X_1]_\theta$ for some $F \in\mathcal{F}$, then $\abs{F}^q$ is Banach-valued for all $q \in (0,p]$ but not analytic in general.
    Using scalar-valued theory, it is shown in \parencite[Thm.~3.4]{InterpolationScalesKaltonMitrea} that $\abs{F}^q$ satisfies the sub-mean value formula
    \begin{align*}
        \abs{F(z)(\spacedot)}^q \leq \dashint_0^{2\pi} \abs{F(z+r\e^{\ii t})(\spacedot)}^q \dd t = \left[\dashint_0^{2\pi} \abs{F(z+r\e^{\ii t})}^q \dd t\right](\spacedot),
    \end{align*}
    whenever $B_r(z)$ is compactly contained in $\strip$.
    We stress that the above is an a.e.\ estimate on $\Omega$ and that the right-hand side is a vector-valued Riemann integral. Interchanging the inner variable and the integral was a non-trivial consequence of analyticity of $F$, compare with \parencite[Lem 7.5]{MasterThesis}.
    
    Applying a positive functional $\varphi$ (of which there are plenty, because the dual space of $(X_0+X_1)^q$ contains a saturated Banach function space by virtue of the weak order unit, see \parencite[Ch.~71, Thm.~4(a)]{Zaanen} and \parencite[Prop.~2.5]{LoristNieraeth}), yields a scalar-valued subharmonic function $\varphi(\abs{F(\spacedot)}^q)$. Using the Poisson formula for subharmonic functions on $\overline{\strip}$, it is shown in \parencite[Thm.~3.4]{InterpolationScalesKaltonMitrea} that
    \begin{align*}
		\varphi(\abs{F(\theta)}^q) &\leq \int_\R P_0(\theta,\ii t)\, \varphi(\abs{F(\ii t)}^q) \dd t +\int_\R P_1(\theta,1+\ii t)\, \varphi(\abs{F(1+ \ii t)}^q) \dd t.
	\end{align*}
	Pulling $\varphi$ out of the integrals by continuity and using positivity, yields
	\begin{align*}
		\abs{F(\theta)}^q &\leq \int_\R P_0(\theta,\ii t) \abs{F(\ii t)}^q \dd t + \int_\R P_1(\theta,1+\ii t) \abs{F(1+ \ii t)}^q  \dd t \\
		&\eqqcolon (1-\theta) f_0(q)^q + \theta f_1(q)^q,
	\end{align*}
     where the functions $f_j(q)^q$ are defined as $X_j^q$-valued Riemann integrals. Interchanging inner variables and integrals is again a tedious task (see the proof of \parencite[Thm.~7.1, Ch.~7.3]{MasterThesis} for details), but it works and Jensen's inequality yields $\abs{F(\theta)}^q \leq (1-\theta)f_0(p)^q + \theta f_1(p)^q$ almost everywhere.
    By letting $q\to0$, it is argued in \parencite[Thm.~3.4]{InterpolationScalesKaltonMitrea} that $\abs{F(\theta)} \leq f_0(p)^{1-\theta}f_1(p)^\theta$ almost everywhere. Finally, the triangle inequality for $X_j^p$-valued integrals (after renorming) yields
	\begin{align*}
		\norm{f_j(p)}_{X_j} &= \left(\norm{\frac{1}{1-\theta}\int_\R P_j(\theta,j+\ii t) \abs{F(j+ \ii t)}^p  \dd t }_{(X_j)^p}\right)^{1/p} \\
    &\leq \left(\frac{1}{1-\theta}\int_\R P_j(\theta,j+\ii t) \norm{\abs{F(j+ \ii t)}^p}_{(X_j)^p} \dd t \right)^{1/p} \\
    &\leq \norm{F}_\mathcal{F}.
	\end{align*}
	
	\begingroup
	\setlength{\emergencystretch}{.5em}
	\printbibliography[
	heading=bibintoc,
	title={Bibliography}
	]
	\endgroup
	
\end{document}